\documentclass[11pt]{article}

\usepackage{stmaryrd}
\usepackage{textcomp}
\usepackage{enumerate}
\usepackage{setspace}
\usepackage{amssymb}
\usepackage{amsthm}
\usepackage{amsmath}
\usepackage{graphicx}
\usepackage{extarrows}
\usepackage{marvosym}
\usepackage{empheq}
\usepackage{latexsym}
\usepackage{endnotes}
\usepackage{natbib}
\usepackage{fontenc}
\usepackage{color}
\usepackage{url}

\newtheorem{theorem}{Theorem}
\newtheorem{lemma}[theorem]{Lemma}
\newtheorem{claim}{Claim}
\newtheorem{corollary}[theorem]{Corollary}

\newtheorem{observation}[theorem]{Observation}
\newtheorem{conjecture}{Conjecture}
\theoremstyle{definition}

\newtheorem*{remark*}{Remark}

\def\AA{{\mathcal A}}
\def\EE{{\mathbb E}}

\def\ZZ{{\mathbb Z}}

\renewcommand{\ge}{\geqslant}
\renewcommand{\geq}{\geqslant}
\renewcommand{\le}{\leqslant}
\renewcommand{\leq}{\leqslant}
\renewcommand{\Pr}{{\mathbb P}}
\newcommand\cC{\mathcal{C}}
\newcommand\cB{\mathcal{B}}
\newcommand\cE{\mathcal{E}}
\newcommand\cF{\mathcal{F}}
\newcommand\cG{\mathcal{G}}
\newcommand\sF{\mathfrak{F}}
\newcommand\cc{\mathrm{c}}

\newcommand\ceil[1]{\lceil #1\rceil}
\newcommand\eps{\varepsilon}

\def\diam{\operatorname{diam}}

\begin{document}

\title{On the maximum running time in graph bootstrap percolation}
\author{B\'ela Bollob\'as%
\thanks{Department of Pure Mathematics and Mathematical Statistics,
Wilberforce Road, Cambridge CB3 0WB, UK and
Department of Mathematical Sciences, University of Memphis, Memphis TN 38152, USA.
E-mail: {\tt b.bollobas@dpmms.cam.ac.uk}.}
\thanks{Research supported in part by NSF grant DMS-1301614 and
EU MULTIPLEX grant 317532.}
,
Micha{\l} Przykucki%
\thanks{Mathematical Institute, University of Oxford, Radcliffe Observatory Quarter, Woodstock Road, Oxford OX2 6GG, UK.
E-mail: {\tt michal.przykucki@st-annes.ox.ac.uk}.}
,
Oliver Riordan%
\thanks{Mathematical Institute, University of Oxford, Radcliffe Observatory Quarter, Woodstock Road, Oxford OX2 6GG, UK.
E-mail: {\tt riordan@maths.ox.ac.uk}.}
\ and Julian Sahasrabudhe%
\thanks{Department of Mathematical Sciences, University of Memphis, Memphis TN 38152, USA.
E-mail: {\tt julian.sahasra@gmail.com}.}
} \maketitle

\begin{abstract}
 Graph bootstrap percolation is a simple cellular automaton introduced by Bollob\'as in
 1968. Given a graph $H$ and a set $G \subseteq E(K_n)$ we initially ``infect'' all edges
 in $G$ and then, in consecutive steps, we infect every $e \in K_n$ that completes a new
 infected copy of $H$ in $K_n$. We say that $G$ \emph{percolates} if eventually every edge in
 $K_n$ is infected. The extremal question about the size of the smallest percolating sets
 when $H = K_r$ was answered independently by Alon, Kalai and Frankl.
Here we consider a different question raised more recently by Bollob\'as: what is
 the maximum time the process can run before it stabilizes?
It is an easy observation that
 for $r=3$ this maximum is $\lceil \log_2 (n-1) \rceil $. However, a new phenomenon occurs
 for $r=4$ when, as we show, the maximum time of the process is $n-3$. For $r \geq 5$ the
 behaviour of the dynamics is even more complex, which we demonstrate by showing that the
 $K_r$-bootstrap process can run for at least $n^{2-\varepsilon_r}$ time steps for some
 $\varepsilon_r$ that tends to $0$ as $r \to \infty$.
\end{abstract}

\section{Introduction}

Graph bootstrap percolation was introduced by Bollob\'as in 1968~\cite{weakSaturation} under the name weak saturation. For a fixed graph $H$, a graph $G$ on $n$ vertices is called \emph{weakly $H$-saturated} if there is no copy of $H$ in $G$, but there is an ordering of the missing edges of $G$ so that if they are added one-at-a-time, every added edge creates a new copy of $H$. Bollob\'as conjectured that the minimum size of a weakly $K_r$-saturated graph is $\binom{n}{2}-\binom{n-r+2}{2}$. This conjecture was confirmed independently by Alon~\cite{alonSaturation}, Frankl~\cite{franklSaturation} and Kalai~\cite{kalaiSaturation}.

Recently, Balogh, Bollob{\'a}s and Morris~\cite{graphBootstrap} observed that weak
saturation is strongly related to bootstrap percolation, a dynamical process introduced by
Chalupa, Leath and Reich~\cite{bootstrapbethe} in 1979 to model the behaviour of
ferromagnets. In bootstrap percolation on a graph $G=(V,E)$ with infection threshold $r
\geq 2$ we choose a set $A \subseteq V$
of initially ``infected'' vertices and we declare
the remaining vertices ``healthy''. Then, in consecutive rounds, we infect all healthy
vertices with at least $r$ infected neighbours. We say that $A$ \emph{percolates} if, starting
from $A$ as the set of the initially infected vertices, we eventually infect every vertex
in $V$. More precisely, we set $A_0 = A$ and for $t=1,2,3, \ldots$ we define
\begin{equation}
\label{eq:bootstrapPercolation}
 A_{t} = A_{t-1} \cup \{v \in V: |N_G(v) \cap A_{t-1}| \geq r \},
\end{equation}
where $N_G(v)$ is the neighbourhood of the vertex $v$ in $G$. Hence $A$ percolates if we have $\bigcup_{t=0}^{\infty} A_t = V(G)$.

Let us then redefine weak saturation in the language of bootstrap percolation. We fix a graph $H$ and we choose a set $G \subseteq E(K_n)$ of edges that we initially infect.
(We shall switch back and forth between thinking of $G$ as a graph and as a set of edges.)
Then, in consecutive steps, we infect every $e \in K_n$ that completes a new infected copy of $H$ in $K_n$. Formally, we take $G_0 = G$ and for $t = 1,2,3, \ldots$ let
\begin{equation}
\label{eq:graphBootstrap}
 G_{t} = G_{t-1} \cup \{e \in E(K_n): G_{t-1}+e \mbox{ contains more copies of } H \mbox{ than } G_{t-1} \}.
\end{equation}
We call this process \emph{$H$-bootstrap percolation} and we say that $G$
\emph{percolates} (or \emph{$H$-percolates}) if eventually every edge in $K_n$ is
infected.
Hence, $G$ is weakly $H$-saturated if and only if it percolates in
$H$-bootstrap percolation and contains no copy of $H$.

In bootstrap percolation we are usually interested in the following setup: for $0 < p < 1$, we initially infect every vertex in $V$ independently at random with probability $p$. We then look for the critical probability $p = p_{\mathrm{c}}(G,r)$ for which percolation is more likely to occur than not. A large number of results, often very sharp, have been obtained in the search for critical probabilities for various graphs $G$ and values of $r$. To name a few, the values of the critical probabilities $p_{\mathrm{c}}(\ZZ^d,r)$ for infinite grids were found by van Enter~\cite{strayleysargument} ($r=d=2$) and Schonmann~\cite{cellularbehaviour} (for general $r$ and $d$). Holroyd~\cite{sharpmetastability} and Balogh, Bollob{\'a}s, Duminil-Copin and Morris~\cite{sharpbootstrapall} obtained sharp bounds on the critical probabilities for finite grids. Various families of random graphs were studied in this context by Balogh and Pittel~\cite{randomregular}, Janson, {\L}uczak, Turova and Vallier~\cite{bootstrapgnp} and Bollob{\'a}s, Gunderson, Holmgren, Janson and Przykucki~\cite{galtonWatson}.

An analogous question in graph bootstrap percolation was considered in~\cite{graphBootstrap}. There, the authors look at the random graph $G(n,p)$ obtained by choosing every edge of $K_n$ independently at random with probability $p = p(n)$. For fixed $r$,
they obtain bounds on the values of $p$ for which $G(n,p)$ is likely to percolate in the $K_r$-bootstrap process.

Another family of problems considered in bootstrap percolation concerns the extremal properties of the process. For example, the size of the smallest percolating sets in various graphs was studied by Pete~\cite{randomdisease}, Balogh, Bollob{\'a}s, Morris and Riordan~\cite{linearAlgebra} and by Morrison and Noel~\cite{smallestHypercube}. Morris~\cite{largestgridbootstrap} and Riedl~\cite{largesthypercubebootstrap} analysed the size of the largest minimal percolating sets for the $n \times n$ square grid and the hypercube respectively, both for $r = 2$. Przykucki~\cite{przykucki01} and Benevides and Przykucki~\cite{benevidesprzykucki02,benevidesprzykucki01} studied the maximum time the infection process on a square grid or a hypercube can take before it stabilizes for $r=2$.

In this paper we study a question due to Bollob\'{a}s~\cite{BBProblemPaper}, who initiated the study of the maximum running time of the graph bootstrap process. Let us define the \emph{maximum running time} of the $K_r$-bootstrap process to be
\begin{equation}
\label{eq:maxRunningTime}
 M_r(n) = \max \{t: \mbox{there exists } G \subseteq E(K_n) \mbox{ s.t. } G_t \neq G_{t-1} \mbox{ in the } K_r \mbox{-bootstrap process} \},
\end{equation}
where $G_t$ is defined as in \eqref{eq:graphBootstrap}. We say that the process \emph{stabilizes at time $t$} if $t$ is the smallest integer such that  $G_{t} = G_{t+1}$. Hence $M_r(n)$ is the maximum time the $K_r$-bootstrap process takes before it stabilizes starting from any initially infected graph with $n$ vertices.

It is an easy observation that $M_3(n) = \lceil \log_2 (n-1) \rceil$.
Indeed, at every step of the $K_3$-bootstrap process, any two vertices at distance two are joined by an edge. Hence, if $G_t$ is connected then we have $\diam(G_{t+1}) = \lceil \diam(G_t) / 2 \rceil$. If $G_t$ is not connected then the components of $G_t$ remain disjoint in $G_{t+1}$. Therefore to maximize the time until the process stabilizes we should take $G_0$ to be a connected graph with as large diameter as possible, i.e., a path on $n$ vertices. It then takes $\lceil \log_2 (n-1) \rceil$ steps until we obtain a clique.

We shall prove two main results.

\begin{theorem}
\label{thm:maxTime4}
We have $M_4(n) = n-3$ for all $n \geq 3$.
\end{theorem}

\begin{theorem}
\label{thm:maxTime5}
For each fixed $r\ge 5$, we have $M_r(n) \geq n^{2-\alpha_r-o(1)}$ as $n\to\infty$,
where $\alpha_r = \frac{r-2}{\binom{r}{2}-2}$.
\end{theorem}
 
The rest of this paper is organised as follows. In Sections~\ref{sec:r=4} and~\ref{sec:r>=5} we prove Theorems~\ref{thm:maxTime4} and~\ref{thm:maxTime5} respectively. We prove Theorem~\ref{thm:maxTime5} by giving a probabilistic argument showing the existence of initially infected sets of edges satisfying our bound. In Section~\ref{sec:deterministic} we give a deterministic construction of an initially infected set that in the $K_5$-bootstrap process runs for at least $\frac{1}{200}n^{3/2}$ time steps before it stabilizes. The construction can be easily adapted for all $r \geq 5$, always guaranteeing the running time to be at least $c_r n^{3/2}$ (this is weaker than Theorem~\ref{thm:maxTime5}, but the proof is much simpler). Finally, in Section~\ref{sec:problems} we discuss some open problems.

\section{Maximum running time in the $K_4$-bootstrap process}
\label{sec:r=4}

In this section we prove Theorem~\ref{thm:maxTime4}, that is, we show that $M_4(n) = n-3$ for all $n \geq 3$. We prove this in two steps: we first give a simple construction to show $M_4(n) \geq n-3$; then we prove the corresponding upper bound. Given a graph $G$, in this section we define $G_t$ as in \eqref{eq:graphBootstrap} with $H = K_4$. Thus $G_t$ is the graph obtained after running the $K_4$-process for $t$ steps, starting with the graph $G$.

\begin{lemma}
For all integers $n\geq 3$, we have $M_4(n) \geq n-3$.
\end{lemma}
\begin{proof}
The case $n=3$ is trivial. We shall show by induction that for each $n\ge 4$ there is an $n$-vertex graph $G$ that stabilizes
in time exactly $n-3$, for which $G_{n-3}=K_n$. The base case is trivial taking $G=K_4^-$, a complete graph on $4$ vertices with one edge deleted.

Suppose then that $n\ge 4$ and $G$ is such a graph; it remains to construct a corresponding graph $G^+$
on $n+1$ vertices. Let $xy$ be any edge of $G_{n-3}\setminus G_{n-4}$, and let $G^+=G+xz+yz$ be formed
from $G$ by adding one new vertex, $z$, and two new edges, $xz$ and $yz$.
Let $t\le n-3$ be the first time that there is a copy of $K_4^-$ present in $G^+_t$ but not in $G_t$.
Until this time, exactly the same edges are added in the percolation processes starting from $G$ and from $G^+$,
so $G_t^+=G_t+xz+yz$. Any $K_4^-$ present in this graph but not in $G_t$ must contain $z$, and hence
$x$, $y$ and one other vertex $w$. Since the edge $zw$ is missing, $xy$ must be present, so $t\ge n-3$.
Conversely, since $G_{n-3}$ is complete, $G_{n-3}^+$ consists of a complete graph on $n$ vertices
with an extra vertex $z$ joined to $x$ and to $y$, and it follows that all edges $zw$ are added in the next
step, so $G_{n-2}^+$ is complete, as required.
\end{proof} 

We now turn to prove the upper bound for Theorem~\ref{thm:maxTime4}
\begin{observation}
\label{obs:cliqueMerge}
Suppose that $K_1,K_2 \subseteq V(G)$ are such that  $K_1,K_2$ induce cliques in $G_t$ with $|K_1 \cap K_2| \geq 2$. Then $K_1 \cup K_2$ induces a clique in $G_{t+1}$.
\end{observation}

To prove that $M_4(n) \geq n-3$, we prove the following stronger statement by induction. Once we have established it, the upper bound will easily follow.

\begin{lemma}
\label{lem:K_4EdgesCompleteCliques}
Let $t \geq 1$ and suppose that $e = ij \in G_{t} \setminus G_{t-1}$ in the $K_4$-bootstrap process. Then for some $d = d(e) \in \{0,1\}$ the edge $e$ is contained in a clique of size $t+3+d$ in $G_{t+d}$.
\end{lemma}
\begin{proof}
We prove the lemma by induction on $t$. The case $t = 1$ is immediate. So assume that $t \geq 2$ and that the statement is true for $t-1$. Now suppose that $e = ij \in G_{t} \setminus G_{t-1}$. We want to find $d(e) \in \{0,1\}$ such that $e$ is contained in a clique of size $t+3+d(e)$ in $G_{t+d(e)}$.

The edge $ij$ appears in $G_{t}$ by virtue of some $4$-element set $K$ such that $i,j \in K$ and $(K^{(2)} - ij) \subseteq G_{t-1}$, where $K^{(2)}$ is the complete graph on $K$.
Note that we cannot have $(K^{(2)} - ij) \subseteq G_{t-2}$, otherwise we would have $ij \in G_{t-1}$. So there is another edge $f$ of $K^{(2)}$ in $G_{t-1} \setminus G_{t-2}$. By induction there exists some $d(f)$ and a set $C \subseteq V(G)$ such that $f \in C^{(2)}$, $|C| \geq (t-1) +3+ d(f)$ and $C^{(2)} \subseteq G_{t-1+d(f)}$. We consider the following cases.
\begin{enumerate}
 \item Suppose first that $i,j \in C$. Then we cannot have $d(f) = 0 $ as this would imply that $C$ induces a clique in $G_{t-1}$ and hence the edge $e$ is in $G_{t-1}$, which contradicts the choice of $e$. Hence we have $d(f)=1$ and, taking $d(e) = d(f) - 1 = 0$, we observe that $C$ induces a clique in $G_{t+d(e)} = G_{t}$ of size $(t-1) + d(f) + 3 = t + d(e)+3 = t+3$ that contains $i,j$. Hence, we are done in this case.
 \item Suppose instead that at least one of the vertices $i,j$ is not in $C$. We further divide this case into two subcases.
 \begin{enumerate}
  \item If $d(f) = 1 $ then $K$ and $C$ both induce cliques in $G_{t-1+d(f)} = G_t$. Since $f \in K^{(2)} \cap C^{(2)}$, we have $|K \cap C| \geq 2 $ and, by Observation~\ref{obs:cliqueMerge}, at time $t + 1$, $K$ and $C$ will merge to form a single clique of size at least $|C|+1\ge t + 3+1$. Hence, we are done choosing $d(e) = 1$.
  \item If $d(f) = 0$ then both $K$ and $C$ induce cliques in $G_{t}$. If $|K \cap C | = 2$ then by Observation~\ref{obs:cliqueMerge} at time $t+1$ they merge to form a complete graph of size $|C|+2 \ge t + 1 + 3$ in $G_{t + 1}$. So again we are done taking $d(e) = 1$.
  
  If $|K \cap C | = 3$ then $C$ contains exactly one endpoint of $e$. Assume that $i \in C, j \notin C$. Since $K^{(2)} - e \subseteq G_{t-1}$, $j$ has two neighbours $j_1, j_2 \in K \cap C$ in $G_{t-1}$. This implies that both $C$ and $\{j, j_1, j_2\}$ induce cliques in $G_{t-1}$ and that they intersect in two points, i.e., in $j_1$ and $j_2$. Hence by Observation~\ref{obs:cliqueMerge} we have that $C \cup \{j\}$ induces a clique of size $|C|+1\ge t + 3$ in $G_{t}$. Hence, taking $d(e)=0$ we are done. This completes the proof of Lemma~\ref{lem:K_4EdgesCompleteCliques}.
 \end{enumerate}
\end{enumerate}
\end{proof}
 
Theorem~\ref{thm:maxTime4} follows easily from Lemma~\ref{lem:K_4EdgesCompleteCliques}. Let $T$ be the time at which the $K_4$-bootstrap process stabilizes starting from an arbitrary graph $G = G_0$. For any edge $e$, let $t(e)$ be the time at which the edge $e$ is added to the infected set. Now for all $e \in G_T$, we have that 
\[
 n \geq t(e) + d(e) + 3
\]
since by Lemma~\ref{lem:K_4EdgesCompleteCliques} we know that $G_{t(e) + d(e)}$, a graph on $n$ vertices, contains a clique of size $t(e)+d(e)+3$. Now since $d(e) \in \{0,1\}$, we have
\[
 n \geq \max_{e \in G_T} \{ t(e) + d(e) + 3 \} \geq \max_{e \in G_T} \{ t(e) + 3 \} = T + 3.
\]
This completes the proof Theorem~\ref{thm:maxTime4}.
 
\section{Lower bounds on $M_r(n)$ for $r \geq 5$}
\label{sec:r>=5}

In this section we use a probabilistic argument to prove Theorem~\ref{thm:maxTime5}. In other words, taking $\alpha_r = \frac{r-2}{\binom{r}{2}-2}$, we show that for every $\varepsilon > 0$ and for $n$ large enough (depending on $r$ and $\eps$) there exists a graph $G = G(r,n)$ such that the $K_r$-bootstrap process started from $G$ stabilizes after at least $n^{2-\alpha_r-\varepsilon}$ time steps.
Throughout we fix $r \geq 5$ and $\eps>0$; all constants in what follows may depend on $r$
and $\eps$.

The simplest possible way to construct a graph that takes a fairly long time
to stabilize in the $K_r$-bootstrap process is as follows. Start with a chain of $K_r$s
that are vertex disjoint except that each shares an edge with the next. More precisely,
consider a sequence $e_0,H_1,e_1,H_2,\ldots,e_{t-1},H_t,e_t$ where each $H_i$ is a copy of $K_r$,
each $e_i$ is an edge, and the entire collection of $e_i$s and $H_j$s
is vertex disjoint except that both $e_{i-1}$
and $e_i$ are (vertex disjoint) edges of $H_i$.
Take $G_0$ to be the union of all the $K_r$s in the chain, with all edges $e_i$ except $e_0$
deleted. Then, clearly, $e_i$ is the unique edge added at step $i$, and the process
stabilizes at time $t$. Of course, this construction only gives the rather weak
lower bound $M_r(n)\ge \lfloor (n-2)/(r-2) \rfloor=\Theta(n)$. 

To obtain a stronger bound, the idea is to relax the disjointness conditions, but in a way that
will not affect the percolation process. More precisely, by a \emph{$K_r$-chain} (within $K_n$) we mean
a sequence $H_1,H_2,\ldots,H_t$ where each $H_s$ is a complete graph on $r$ vertices (contained
within $K_n$), and $H_i$ and $H_j$ are edge-disjoint unless $|i-j|=1$, in which case they
share exactly one edge. Given such a chain we let $e_i$ be the edge shared by $H_i$ and $H_{i+1}$,
let $e_0$ be some edge of $H_1$ other than $e_1$, and let $e_t$ be some edge of $H_t$
other than $e_{t-1}$. Sometimes we describe the chain by the list
$e_0,H_1,e_1,H_2,\ldots,e_{t-1},H_t,e_t$. In this list, the edges $e_i$ are distinct, and $e_i$ is an edge
of $H_j$ if and only if $i\in\{j-1,j\}$.

Let $K_r^-$ denote the graph formed by deleting an edge from $K_r$.
Given a $K_r$-chain as above,
an \emph{external $K_r^-$} is a (not necessarily induced) subgraph of the graph union
$\bigcup_{i=1}^t H_i$ which is isomorphic to $K_r^-$ and is not contained in any individual $H_i$.
A $K_r$-chain is \emph{good} if it generates no external $K_r^-$.

\begin{observation}\label{obsgc}
If $e_0,H_1,e_1,H_2,\ldots,e_{t-1},H_t,e_t$ is a good $K_r$-chain, then the graph $G_0=\bigcup_{i=1}^t H_i
-\{e_1,e_2,\ldots,e_t\}$ stabilizes at time $t$ in the $K_r$-bootstrap process.
\end{observation}
\begin{proof}
Although this is immediate, let us spell out the details.
Define $G_t$ as in \eqref{eq:graphBootstrap}, with $H=K_r$. We claim that for $i\le t$
we have $G_i=G_0\cup\{e_1,\ldots,e_i\}$. Clearly this holds for $i=0$. Suppose it holds
for some $0\le i\le t-1$; then $G_i$ is a subgraph of $\bigcup_{j\le t}H_j$, so any copy of $K_r^-$
in this graph has vertex set $V(H_j)$ for some $j$. Within $V(H_j)$, exactly two
edges are missing in $G_0-e_0$, namely $e_{j-1}$ and $e_j$. Since (by induction) $G_i$
contains $e_k$ if and only if $k\le i$, we have an induced copy of $K_r^-$ if and only if $j=i+1$,
so $e_{i+1}$ is the unique edge added at step $i+1$. It follows that $G_0$ stabilizes at time $t$.
\end{proof}

Our aim is to show that we can fit a very long good $K_r$-chain into a set of $n$ vertices, i.e.,
to prove the following result. Here and in what follows we ignore rounding to integers when
it makes no essential difference.

\begin{theorem}\label{th2}
Let $r\ge 5$ and $\eps>0$ be given, and set $\alpha_r= \frac{r-2}{\binom{r}{2}-2}$.
If $n$ is large enough, then there is a good $K_r$-chain $e_0,H_1,e_1,\ldots,e_{T-1},H_T,e_T$
with $T = n^{2 - \alpha_r - \varepsilon}$ using at most $n$ vertices in total.
\end{theorem}

By Observation~\ref{obsgc} above, Theorem~\ref{th2} implies 
Theorem~\ref{thm:maxTime5}. The rest of this section is devoted 
to the proof of Theorem~\ref{th2}, for which we use a random construction. We outline
this informally before giving the formal proof. Throughout, $r\ge 5$ and $\eps>0$ are fixed,
$n$ is sufficiently large for various conditions that arise below to hold,
and $\alpha_r$ and $T$ are defined as in Theorem~\ref{th2}.

We will construct the $K_r$-chain randomly. There is a snag, however: if we follow
the obvious method (just pick $r-2$ new random vertices each time) then because there are many steps,
we are likely to get stuck (create an external $K_r^-$) relatively early, at a point where
each individual step is unlikely to cause a problem. Intuitively, we should
be able to keep going as long as there is a decent chance that the next step succeeds.
To make this precise, we consider a number $m=\ceil{\log n}$ of attempts at choosing
the next $r-2$ vertices, and continue if one of these succeeds.
Of course this leads to a lot of dependence (which choice succeeds depends on what happened at previous
steps). But we can get around this in the analysis by considering \emph{all possible} ways
that our choices could lead to an external $K_r^-$.

\medskip
\noindent {\bf Outline construction.}
Let $e_0$ be a uniformly random edge in $K_n$.
For $1\le t\le T$ and $1\le a\le m$, let $X_{t,a}$ be a uniformly random set of $r-2$
vertices of $K_n$, and $e_{t,a}$ a uniformly random edge within $X_{t,a}$,
with these choices independent over all $t$ and $a$.
For $t$ running from $0$ to $T-1$ do the following: pick an index $a=a_{t+1}$ so that certain
conditions specified later hold (if this is possible), set $X_{t+1}=X_{t+1,a}$ and 
$e_{t+1}=e_{t+1,a}$, and let $H_{t+1}$ be the complete graph on $e_t\cup X_{t+1}$.

\medskip
We think of $X_{t+1,a}$ as the $a$th \emph{candidate} set of new vertices in step $t+1$,
and $X_{t+1}$ as the selected set; similarly, $e_{t+1,a}$ is a candidate next edge
in the chain, and $e_{t+1}$ the actual next edge. A candidate set $X_{t+1,a}$
is \emph{successful} if it satisfies certain conditions \eqref{C0}--\eqref{C4} below.
Later we shall write $\cF_{t+1}$ for the
event that our construction \emph{fails} at step $t+1$, in the sense
that none of the candidates $X_{t+1,a}$ for $X_{t+1}$ is successful. We write
\[
 \cG_t = \cF_1^\cc\cap \cdots \cap\cF_t^\cc
\]
for the `good' event that the construction succeeds up to (at least) step $t$.

The first of the conditions referred to above is
\[
 \text{$X_{t+1}=X_{t+1,a_t}$ is disjoint from $e_t$.} \tag{C0}\label{C0}
\]
This condition ensures that $e_t\cup X_{t+1}$ is indeed a set of $r$ vertices.

We shall always write $G_t=\bigcup_{1\le s\le t} H_s$ for the graph constructed so far.
By a \emph{new $i$-set} of vertices we mean a set of $i\ge 2$ vertices of $H_{t+1}$
not all contained in $e_t$. Note that any `new' edge in the chain (i.e., edge present in $G_{t+1}$
but not in $G_t$) is a new 2-set. Our second condition is the following:
\[
 \text{No new $2$-set is contained in any $H_s$, $s\le t$.} \tag{C1}\label{C1}
\]
If conditions \eqref{C0} and \eqref{C1} hold at each step $t'\le t$, then $e_0,H_1,e_1,\ldots,H_t,e_t$
is a $K_r$-chain: each $H_s$ is a complete graph on $r$ vertices, and for $i<j$ (using \eqref{C1} at step $j$)
we see that $H_i$ and $H_j$ are edge-disjoint unless $j=i+1$, in which case they share
only one edge, $e_i$.

Given a set $A$ of vertices, we write $K_A$ for the complete graph with vertex set $A$,
i.e., the graph $(A,A^{(2)})$.
For $2\le i\le r-1$, we say that a set $A$ of $i$ vertices is \emph{dangerous at time $t$}
(or sometimes just \emph{dangerous})
if there is a set $B$ of $r$ vertices with $A\subset B$ and $B\ne V(H_s)$ for all $s\le t$
such that the graph $G_t\cup K_A$
includes all but at most one of the edges within $B$. In other words, $A$ is dangerous if adding
all edges within $A$ would create an external $K_r^-$. The next condition involves avoiding
such sets $A$:
\[
 \text{For $2\le i\le r-1$, no new $i$-set in $e_t\cup X_{t+1,a_t}$ was dangerous at time $t$.}\tag{C2}\label{C2}
\]
\begin{observation}
If, for each step $t$, $1\le t\le T$, we manage to choose a candidate $X_t=X_{t,a}$
so that conditions \eqref{C0}--\eqref{C2} hold, then $(e_0,H_1,\ldots,H_T,e_T)$ is a good
$K_r$-chain.
\end{observation}
\begin{proof}
As noted above, \eqref{C0} and \eqref{C1} ensure that we get a $K_r$-chain.
Suppose it is not good. Then
at some step $t$ an external $K_r^-$ must have been generated. Let its vertex set be $B$.
Clearly, $G_t\setminus G_{t-1}$ contains at least one edge within $B$. Hence $A=B\cap V(H_t)$
is a new $i$-set for some $i\ge 2$; also $i<r$ by definition of an external $K_r^-$.
But then $A$ was dangerous at step $t-1$, so condition \eqref{C2} did not hold.
\end{proof}

If steps $1,2,\ldots,t$ succeed, i.e., $\cG_t$ holds, then by definition conditions \eqref{C0}--\eqref{C2}
were satisfied at steps $1,\ldots,t$. Hence
\[
 \cG_t \text{ implies that $H_1,\ldots,H_t$ is a good $K_r$-chain}.
\]
Our aim is to show that with positive probability, we can satisfy conditions \eqref{C0}--\eqref{C2}
above. Since our construction is inductive, when considering the candidates for $X_{t+1}$ we
may assume that $\cG_t$ holds.

A key step in our proof is to show that there are not too many dangerous $i$-sets. More
generally, we shall show that for any graph $F$ with $r$ vertices, if we exclude
copies arising inside a single $H_s$, it is unlikely that $G_t$ contains many more
copies of $F$ than a random graph with the same overall edge density as $G_t$.
To show this, we fix a set $A$ of $r$ vertices, and think about all possible ways
that the chain $H_1,H_2,\ldots,H_t$ can meet $A$, considering only cliques $H_s$
that contribute an edge within $A$, i.e., satisfy $|V(H_s)\cap A|\ge 2$.
Since successive $K_r$s in the chain are related in a different way from those further apart,
we shall split the set of $K_rs$ in the chain that meet $A$ into subchains. This (hopefully)
motivates the following definition.

Let $A$ be a set of $r$ vertices. By a \emph{chain within $A$} we mean a sequence
$(S_1,\ldots,S_k)$ of subsets of $A$ such that
\smallskip

(i) $2\le |S_i|\le r-1$ for $1\le i\le k$ and

\smallskip
(ii) $|S_i\cap S_{i+1}|\le 2$ for $1\le i\le k-1$.

\smallskip\noindent
By a (partial) \emph{chain cover} $\cC$ of $A$ we mean a sequence
$(S_1^{1},\ldots, S_{k_1}^1)$, $(S_1^2,\ldots, S_{k_2}^2), \ldots , (S^{\ell}_1,\ldots, S^\ell_{k_\ell})$
of chains within $A$.
Such a chain cover is \emph{minimal} if every $S_i^j$ spans some edge not spanned by any other
set $S_{i'}^{j'}$ in the chain cover. Clearly, a minimal chain cover contains at most
$\binom{r}{2}$ sets $S_i^j$ in total, so there are a finite number (depending on $r$) of 
minimal chain covers of any $r$-set $A$. Also, if $\cC$ is minimal,
then each of its constituent chains is minimal in the natural sense (i.e.,
as a chain cover with one chain).

The \emph{cost} of a chain $C=(S_1,\ldots,S_k)$ is
\[
 c(C) = |S_1| + \sum_{i=1}^{k-1} |S_{i+1}\setminus S_i|,
\]
informally corresponding to the number of `new' vertices in the chain.
The \emph{edge set} $E(C)$ of $C$ is simply $\bigcup_{i=1}^k S_i^{(2)}$.
The \emph{benefit} $b(C)$ of $C$ is $|E(C)|$.
For a chain cover $\cC=(C_1,\ldots,C_j)$ we define its cost and benefit by
$c(\cC)=\sum c(C_i)$ and $b(\cC)=\sum b(C_i)$. (For $b(\cC)$ it might be more natural
to take the total number of edges without repetition;
this makes no difference in the argument below.)

Given a set $A$ of $r$ vertices and a chain cover
$\cC=((S_1^{1},\ldots, S_{k_1}^1), (S_1^2,\ldots, S_{k_2}^2), \ldots , (S^{\ell}_1,\ldots, S^\ell_{k_\ell}))$ of $A$,
we say that our random $K_r$-chain $H_1,\ldots,H_t$ \emph{meets $A$ in $\cC$} if there
exist $0\le t_1<\cdots<t_\ell$ such that
\smallskip

$t_j+k_j \le t_{j+1}-1$ for each $1\le j\le \ell-1$ and $t_\ell+k_\ell\le t$, and

\smallskip
$V(H_{t_j+i})\cap A=S_i^j$ for each $j\le \ell$ and $1\le i\le k_j$.

\smallskip\noindent 

In other words, our $K_r$-chain contains $\ell$ intervals of consecutive $K_r$s, the $j$th running from time $t_j+1$
to time $t_j+k_j$,
with gaps between the intervals, so that in each interval the intersections of the $K_r$s with $A$ form
the $j$th chain of $\cC$. (We impose no condition on how other cliques $H_s$ may meet $A$.)

\begin{observation}\label{omeet}
Suppose that $\cG_t$ holds. Let $A$ be a set of $r$ vertices
spanning at least $e$ edges in $G_t=\bigcup_{s\le t} H_s$,
with $A\ne V(H_s)$ for each $s\le t$.
Then there is some minimal chain cover $\cC$ with $b(\cC)\ge e$ such that $(H_s)$ meets $A$ in $\cC$.
\end{observation}
\begin{proof}
Simply consider a minimal subset of the cliques $H_s$ that between them contain all
$\ge e$ edges of $G_t$ inside $A$, and split this set of cliques into intervals with gaps of
at least one between them. Each of these $H_s$ shares at least two and at most $r-1$ vertices with $A$,
and consecutive $H_s$ meet in at most two vertices, since $\cG_t$ holds so
$(H_1,\ldots,H_t)$ is a $K_r$-chain.
\end{proof}

Lemma~\ref{bc1} and its consequence Corollary~\ref{bcl} below are key to our analysis;
the latter says, roughly speaking,
that the `easiest' way for a random $K_r$-chain of length $T$
to cover a set $A$ to form an external $K_r^-$ (or some other graph $F$) is by meeting
it $\binom{r}{2}-1$ times ($e(F)$ times) in non-consecutive cliques $H_i$, each time in just one edge.
This condition is loosely
analogous to the `2-balanced' condition that appears in many small subgraph problems.
We start by showing that for a single chain, a certain `cost-benefit ratio' is maximized
in the single-edge case.

\begin{lemma}\label{bc1}
Let $A$ be a set of $r$ vertices, and $C=(S_1,\ldots,S_k)$ a minimal chain within $A$. Then 
$\alpha_r (b(C)-1) \le c(C)-2$.
\end{lemma}
\begin{proof}
Recalling the definition of $\alpha_r$, and writing $b=b(C)$ and $c=c(C)$, the inequality
claimed can be rewritten as
\begin{equation}\label{bcaim}
 (r-2) (b-1) \le \left(\tbinom{r}{2}-2\right)(c-2).
\end{equation}
Now $|\bigcup S_i|\le c$, and $b = |\bigcup S_i^{(2)}| \le |(\bigcup S_i)^{(2)}|$,
so $b\le \binom{c}{2}$.
The inequality
\[
 (r-2)\left(\tbinom{c}{2}-1\right) \le \left(\tbinom{r}{2}-2\right)(c-2)
\]
is an equality for $c=2$, and is easily seen (by multiplying out) 
to hold (strictly) for $c=r-1$.
Hence (since the left-hand side is convex and the right-hand side linear)
it holds for $2\le c\le r-1$; since $b\le \binom{c}{2}$ this proves \eqref{bcaim} in these
cases.

For $c\ge r+1$ we use the trivial bound $b\le \binom{r}{2}$; it is easy to check
that this implies \eqref{bcaim} in this case. This leaves only the case $c=r$.
In this case \eqref{bcaim} reduces to $b\le \binom{r}{2}-1$, i.e., we must show that
if $c(C)=r$ then there is at least one edge in $A^{(2)}$ missing from $E(C)$.
To see this, recall first that $|S_i|\le r-1$ by definition, so $k\ge 2$.
Also, if $|\bigcup S_i|<c=r$, then clearly $b(C)\le \binom{r-1}{2}<\binom{r}{2}-1$.
Hence, we may assume that $\bigcup S_i=A$ and that only consecutive sets $S_i$
meet. By minimality there is some
$u\in S_1\setminus S_2$, and some $v\in S_k\setminus S_{k-1}$. But then $uv\notin E(C)$
and we are done.
\end{proof}

\begin{remark*}
We don't need it here, but in fact we have strict inequality in the result above unless
$k=1$, $c=2$ and so $b=1$. To see this we must show in the final case above that
at least two edges are missing, which is not hard.
\end{remark*}

\begin{corollary}\label{bcl}
Let $\cC$ be a minimal chain cover of some $r$-set $A$ consisting of $\ell$ chains
and having total cost $c$ and benefit $b$. Then $T^\ell n^{-c} \le n^{-\alpha_r b -\eps}$.
\end{corollary}
\begin{proof}
If $\cC=(C_1,\ldots,C_\ell)$ then $c=\sum_{i=1}^\ell c(C_i)$ and
$b=\sum_{i=1}^\ell b(C_i)$. Applying Lemma~\ref{bc1} to each chain $C_i$ and summing,
we see that $\alpha_r(b-\ell) \le c-2\ell$. Rearranging gives
\[
 (2-\alpha_r)\ell - c \le -\alpha_r b.
\]
The result follows recalling that $T=n^{2-\alpha_r-\eps}$ and noting that $\ell\ge 1$.
\end{proof}

Fix a set $A$ of $r$ vertices. Given a minimal
chain $C=(S_1,\ldots,S_k)$ within $A$ and an integer $0\le t\le T-k$,
let
\[
 [A,t,C]
\]
denote the event that there are indices $a_t,\ldots,a_{t+k}$ such that if we set
$V_s=e_{s-1,a_{s-1}}\cup X_{s,a_s}$ for $s=t+1,\ldots,t+k$, then we have $V_{t+j}\cap A=S_j$
for $j=1,\ldots,k$. In other words, $[A,t,C]$ is the event that it is conceivable
(considering all \emph{a priori} possible choices for which candidate set/edge is chosen at 
each step) that the consecutive $K_r$s $H_{t+1},\ldots,H_{t+k}$ meet $A$ in the chain $C$.
Recalling the definition of $c(C)$, we see that for this event to hold, we need $c$ of the relevant
candidate vertices to fall in $A$.
Taking a union bound over the choices of $a_t,\ldots,a_{t+k}$, and noting that (by minimality),
$k$ is bounded (by $\binom{r}{2}$, say), we see that
\begin{equation}\label{PAtC}
 \Pr([A,t,C]) \le m^{k+1} M_{k,r} n^{-c(C)} = O^*(n^{-c(C)})
\end{equation}
for some constant $M_{k,r}$ depending on $k$ and $r$. Here, as usual, the $O^*$ notation hides
powers of $\log n$ that (with $r$ fixed) are bounded.

For $\cC=(C_1,\ldots,C_\ell)$ a minimal
chain cover of $A$ with each chain $C_j$ having length $k_j$,
let $\cB_{A,\cC}$ be the event that there exist $0\le t_1<t_2<\cdots <t_\ell$ such that
$t_j+k_j \le t_{j+1}-1$ for each $1\le j\le \ell-1$ and $t_\ell+k_\ell\le T$, and
\begin{equation}\label{AtC}
 \text{$[A,t_j,C_j]$ holds for $j=1,2,\ldots,\ell$.}
\end{equation}
Note that if the $K_r$-chain $(H_s)_{s=1}^t$ meets $A$ in $\cC$, then $\cB_{A,\cC}$ holds.
Indeed, informally speaking,
$\cB_{A,\cC}$ is the event that is it conceivable that our chain might meet $A$ in $\cC$.
The advantage of working with $\cB_{A,\cC}$ is that it avoids the complicated dependence introduced
by the rule for selecting which attempt at the next $K_r$ becomes the actual next $K_r$.

Since $[A,t_j,C_j]$ depends only on the random variables $X_{t,a}$, $t_j\le t\le t_j+k_j$, $1\le a\le m$,
with $\cC$ and the $t_j$ fixed the $\ell$ events appearing in \eqref{AtC} are independent.
Hence, using \eqref{PAtC} and summing over the at most $T^\ell$ choices for $t_1,\ldots,t_\ell$,
we have
\[
 \Pr(\cB_{A,\cC}) = O^*(T^\ell n^{-c(C_1)-c(C_2)\cdots-c(C_\ell)}) = O^*(T^\ell n^{-c(\cC)}) = O^*(n^{-\alpha_r b(\cC) -\eps}),
\]
where the final step follows from Corollary~\ref{bcl}.

For $1\le e\le \binom{r}{2}$,
let $\cB_{A,e}$ be the event that there exists some minimal chain cover of $A$ with $b(\cC)\ge e$
such that $\cB_{A,\cC}$ holds.
Since there are $O(1)$ minimal chain covers of $A$, we see that
\begin{equation}\label{PBAe}
 \Pr(\cB_{A,e}) = O^*(n^{-\alpha_r e -\eps}).
\end{equation}
By Observation~\ref{omeet}, for any $t\le T$, if $\cG_t$ holds
and $A$ spans at least $e$ edges in $G_t=\bigcup_{s\le t} H_s$,
with $A\ne V(H_s)$ for each $s\le t$,
then $\cB_{A,e}$ holds. This allows us to prove our key lemma on the number of dangerous $i$-sets.

Recall that a set $A$ of $i$ vertices is dangerous at time $t$ if adding all edges inside $A$
to $G_t$ would create an external copy of $K_r^-$ (or of $K_r$) with vertex set $B\supset A$.
We will consider the `bad' event
\[
 \cB^1_t = \cG_t\cap\{ \text{ for some $2\le i\le r-1$,
 there are more than $n^{i-\eps/2}$ dangerous $i$-sets at time $t$. }\}
\]
\begin{lemma}\label{B1}
For $r\ge 5$ and $\eps>0$ fixed, setting $T=n^{2-\alpha_r-\eps}$, we have
$\Pr(\bigcup_{t\le T} \cB^1_t) = o(1)$.
\end{lemma}
\begin{proof}
For $2\le i\le r-1$ let $Y_{t,i}$ denote the number of dangerous $i$-sets at time $t$.
Let $Z_i$ denote the number of $r$-sets $A$ such that $\cB_{A,e_i}$ holds,
where $e_i=\binom{r}{2}-1-\binom{i}{2}$.
If $\cG_t$ holds and
an $i$-set $A'$ is dangerous at time $t$, then it is contained in an $r$-set $A$, not the vertex
set of any $H_s$, $s\le t$, such that $A$ spans at least $e_i$ edges
in $G_t$. But then as noted above $\cB_{A,e_i}$ holds. Since a particular $r$-set can be responsible
for at most $\binom{r}{i}\le 2^r$ $i$-sets being dangerous, it follows that, for any $t\le T$,
if $\cG_t$ holds then $Y_{t,i}\le 2^r Z_i$. Thus it suffices to show that
\begin{equation}\label{Zaim}
 \Pr(Z_i\ge 2^{-r}n^{i-\eps/2})=o(1)
\end{equation}
for each $i$.

We claim that, for $2\le i\le r-1$, we have
\begin{equation}\label{ri}
 r-i \le \alpha_r\left( \tbinom{r}{2}-1-\tbinom{i}{2} \right).
\end{equation}
This holds for $i=2$ by definition of $\alpha_r$. For $i=r-1$ it simplifies to $1\le \alpha_r(r-2)$
and hence to $(r-2)^2\ge \binom{r}{2}-2$, which is easily seen to hold for $r\ge 5$.
Hence, by convexity, \eqref{ri} holds for all $2\le i\le r-1$.
By \eqref{PBAe} we have
\[
 \EE[Z_i] = \sum_A \Pr(\cB_{A,e_i}) = O^*(n^{r-\alpha_r e_i-\eps}) = o( n^{r-\alpha_r e_i -\eps/2} ).
\]
But then, by \eqref{ri}, $\EE[Z_i]=o(n^{i-\eps/2})$, and \eqref{Zaim} follows by Markov's inequality.
\end{proof}

Lemma~\ref{B1} shows that one `global bad event' is unlikely. The next lemma considers another, much simpler
one.
\begin{lemma}\label{B2}
Suppose that $r\ge 5$ and $\eps>0$, and define $T=n^{2-\alpha_r-\eps}$ as usual.
Let $\cB_t^2$ be the event that the maximum degree of $G_t$ is at least $n^{1-\eps}$.
Then $\Pr(\bigcup_{t\le T} \cB_t^2) =o(1)$.
\end{lemma}
\begin{proof}
Suppose that a vertex $v$ has degree at least $n^{1-\eps}$ in some $G_t$. Then
$v$ must be in (crudely) at least $n^{1-\eps}/r$ of the cliques $H_s$, $s\le t$, and hence
in at least $n^{1-\eps}/(2r)$ of the selected sets $X_s$. Thus, $v$ is in at least
$n^{1-\eps}/(2r)$ of the candidate sets $X_{s,a}$, $0\le s\le T$, $1\le a\le m$.
For each $v$, let $N_v$ be the number of these candidate sets that contain $v$.
Then $N_v$ has a binomial distribution with mean $m(T+1)(r-2)/n=O^*(n^{1-\alpha_r-\eps})=o(n^{1-\eps}/(2r))$.
Thus the probability that $N_v\ge n^{1-\eps}/(2r)$ is $\exp(-\Omega(n^{1-\eps}))=o(n^{-100})$, say,
and the probability that there exists such a $v$ is $o(1)$.
\end{proof}

After this preparation we are now ready to prove Theorem~\ref{th2}.
\begin{proof}[Proof of Theorem~\ref{th2}]
As above, we fix $r\ge 5$ and $\eps>0$, and define $m=\ceil{\log n}$ and $T=n^{2-\alpha_r-\eps}$,
where $\alpha_r=(r-2)/(\binom{r}{2}-2)$. We consider the outline construction described above, 
saying that a candidate choice $X_{t+1,a}$ for $X_{t+1}$ is \emph{successful} if it satisfies
conditions \eqref{C0}--\eqref{C2} above, and \eqref{C3} and \eqref{C4} below.
At each step $t+1\le T$ we choose a successful candidate for $X_{t+1}$ if there is one, otherwise,
our construction \emph{fails at step $t+1$}.

Let $\sF_t$ be the (finite, of course) $\sigma$-algebra generated by $(X_{s,a})_{s\le t,\,a\le m}$,
i.e., by all information `revealed' by time $t$. Our key claim is that for $0\le t\le T-1$ we have
\begin{equation}\label{mainaim}
 \Pr(\hbox{ candidate $X_{t+1,1}$ succeeds }\mid \sF_t) \ge 99/100
\end{equation}
whenever all previous steps have succeeded, and neither $\cB_t^1$ nor $\cB_t^2$ holds. Suppose for the
moment that this holds; we shall see that Theorem~\ref{th2} easily follows.
Indeed, because the $X_{t+1,a}$ are independent of each other and of $\sF_t$, and identically
distributed, if \eqref{mainaim} holds, then recalling that $\cF_{t+1}$ denotes the event that
our construction fails at step $t+1$, we have
\[
 \Pr( \cF_{t+1} \mid \sF_t) = \prod_{a=1}^m  \Pr(\hbox{ candidate $X_{t+1,a}$ fails }\mid \sF_t)
 \le 100^{-m} \le n^{-2}
\]
whenever $(\cB_t^1\cup \cB_t^2)^\cc\cap \cG_t = (\cB_t^1\cup \cB_t^2\cup \cF_1\cup\cdots\cup \cF_t)^\cc$ holds.
Since this latter event is $\sF_t$-measurable, it follows that
\[
   \Pr\bigl( \cF_{t+1} \cap (\cB_t^1\cup \cB_t^2\cup \cF_1\cup\cdots\cup \cF_t)^\cc \bigr) \le 
  \Pr\bigl( \cF_{t+1} \mid (\cB_t^1\cup \cB_t^2\cup \cF_1\cup\cdots\cup \cF_t)^\cc \bigr) \le n^{-2}.
\]
But then, considering the least $t$ for which $\cB^1_t \cup \cB^2_t \cup \cF_t$ holds, we see that
\[
 \Pr \bigl ( \bigcup_{t\le T} \cF_t \bigr) \le Tn^{-2}+\Pr \bigl (\bigcup_{t\le T} \cB_t^1 \bigr)  +\Pr \bigl (\bigcup_{t\le T} \cB_t^2 \bigr) 
 =o(1),
\]
using Lemmas~\ref{B1} and~\ref{B2}. Hence, with probability $1-o(1)>0$ (for $n$ large enough),
the construction succeeds for $T$ steps. Then, as noted above, the fact that conditions
\eqref{C0}--\eqref{C2} are satisfied implies that we construct a good $K_r$-chain of the required length.
It remains only to establish \eqref{mainaim}.

From now on, we condition on $\sF_t$ (i.e., on all $X_{s,a}$, $s\le t$), and assume that our construction
succeeded at steps $1,\ldots,t$ and that neither $\cB_t^1$ nor $\cB_t^2$ holds. The only relevant
randomness remaining is the uniform choice of $X_{t+1,1}$ from all sets of $r-2$ vertices.
Let $\cE_i$ be the event that $X_{t+1,1}$ fails to satisfy condition (C$i$).
To establish \eqref{mainaim} it suffices to show that $\Pr(\cE_i)=o(1)$, $i=0,\ldots,4$.
Clearly, $\Pr(\cE_0)=1-\binom{n-2}{r-2}/\binom{n}{r-2}=o(1)$.

The case $i=1$ is also easy. If $\cE_1$ holds then either (i) $X_{t+1,1}$ includes two vertices
forming an edge $uv$ of $G_t$ or (ii) for some vertex $u$ of $e_t$, $X_{t+1,1}$ includes
a vertex $v$ such that $uv\in E(G_t)$.  The first
event has probability at most $\binom{r-2}{2} e(G_t)/\binom{n}{2} =O(e(G_t)n^{-2})= O(T n^{-2})=o(1)$.
Since $\cB_t^2$ does not hold and there are only two choices for $u$, the second
has probability $O(\Delta(G_t)/n)=o(1)$.

Before turning to the details, let us outline the argument for $\cE_2$. If $\cE_2$ holds,
then for some $0\le j\le 2$ and $1\le i\le r-2$ with $i+j<r$, some subset of $V(e_t)$ of size
$j$ and some subset of $X_{t+1,1}$ of size $i$ combine to form a dangerous set of size $i+j$.
The case $j=0$ will cause no problems, since we assume $\cB_t^1$ does not hold,
so there are $o(n^{i})$ dangerous
sets of size $i$. (This is analogous to (i) above.) For $j>0$, we will have a problem
only if some subset $J$ of $V(e_t)$ of size $j$ is in $\Theta(n^i)$ dangerous $(i+j)$-sets.
To avoid this, we consider a further condition
\[
 \parbox{13cm}{For each $j=1,2$ and $i\ge 1$ with $i+j<r$,
 no set of $j$ vertices of $X_{t+1,a}$ is contained in more
 than $n^{i-\eps/4}$ dangerous $(i+j)$-sets at time $t$.} \tag{C3}\label{C3}
\]
This is not quite what we need, however; we would like to use this condition one step earlier
to say that $e_t\subset X_t$ has the properties we want. The problem is that condition \eqref{C3}
having been satisfied by $X_t$ only tells us
that subsets of $e_t$ are not in many sets that were dangerous \emph{at time $t-1$},
rather than at time $t$. On the other hand, the bound $n^{i-\eps/4}$ in condition \eqref{C3}
is much stronger than the bound $o(n^i)$ that we need. So we will be able to deal with this
`off-by-one' problem 
by showing that not too many sets become dangerous in one step. To make this argument work
we will need one final condition. We say that a set $A$ of vertices
with $2\le |A|\le r-1$ is \emph{deadly} at time $t$
if there is a set $B$ of $r$ vertices with $A\subset B$ such that $G_t\cup K_A$ contains
all edges within $B$, with $B$ not equal to $V(H_s)$ for any $s\le t$.
Note that a deadly set is dangerous.
\[
 \parbox{13cm}{For $j\in\{2,3\}$, no set of $j$ vertices of $e_t\cup X_{t+1,a}$ containing at most
one vertex of $e_t$ and at least one vertex of $X_{t+1,a}$ is contained in more than $n^{1-\eps/8}$ deadly
$(j+1)$-sets at time $t$.} \tag{C4}\label{C4}
\]

As noted above, it remains only to show that $\Pr(\cE_i)=o(1)$ for $i=2,3,4$. Before doing this, we
establish some deterministic consequences of our conditions \eqref{C0}--\eqref{C4}, i.e., of succeeding
with our choice at step $t+1$. In the following we assume that $n$ is large enough, depending
on the constant~$c$.

\begin{claim}\label{claim1}
Let $c>0$ be any positive constant.
If the candidate $X_{t+1,1}$ satisfies conditions \eqref{C0}--\eqref{C4},
then for $j=1,2$ and each $i\ge 1$ with
$i+j<r$, every set $J$ of $j$ vertices from $X_{t+1,1}$ is in at most $c n^i$ $(i+j)$-sets
that are dangerous {\bf at time} $\mathbf{t+1}$.
\end{claim}
\begin{proof}
Suppose the claim fails for some $i$, $j$ and $J$. Because condition \eqref{C3}
is satisfied, we see that there are $\Theta(n^i)$ $i$-sets $I$ such that $I\cup J$ is dangerous
at time $t+1$ but not at time $t$. Since $H_{t+1}=K_{e_t \cup X_{t+1,1}}$
contains $r=O(1)$ vertices and $\Delta(G_{t+1})\le \Delta(G_t)+r-1=o(n)$ by our assumption
that $\cB_t^2$ does not hold,
there are $\Theta(n^i)$ of these sets $I$ with the additional property that no vertex in $I$
is in $H_{t+1}$, or is adjacent in $G_{t+1}$ to any vertex in $H_{t+1}$.
For each such set $I$, by the definition of dangerous there is an $r$-set $A\supset I\cup J$
such that $G_{t+1}\cup K_{I\cup J}$ contains all but at most one of the edges within $A$.
Furthermore, since $I\cup J$ is not dangerous at time $t$, there is such a set $A$ such that
$G_{t+1}\setminus G_t$ includes at least one edge within $A$ but not within $I\cup J$.
In particular, $A\setminus (I\cup J)$ contains at least one vertex of $H_{t+1}$.
Let $L=V(H_{t+1})\cap (A\setminus (I\cup J))$, so $L\ne\emptyset$.

By choice of $I$, there are no edges from $I$ to $H_{t+1}$ in $G_{t+1}$. Hence
there are no edges from $I$ to $L$ in $G_{t+1}$. Since $I$, $J$ and $L$ are disjoint,
in $G_{t+1}\cup K_{I\cup J}$ there are still no edges from $I$ to $L$. Hence, in this final
graph, there are at least $|I||L|$ missing edges inside $A$; as at most one edge is missing,
we conclude that $|I|=|L|=1$. Furthermore, the only missing edge in $A$ in $G_{t+1}\cup K_{I\cup J}$
is the $I$--$L$ edge. Hence in $G_{t+1}\cup K_{I\cup J\cup L}$ the set $A$ is complete.
Since $V(H_{t+1})\cap A=J\cup L$, we see that $A$ is also complete in $G_t\cup K_{I\cup J\cup L}$,
so $I\cup J\cup L$ is deadly \emph{at time $t$}. Note that $|I\cup J\cup L|=j+2<r$.

Since we find one such deadly set $I\cup J\cup L$ for each of $\Theta(n)$ choices for $I$ (recall
that now $i=1$),
and there are at most $r$ choices for $L$, we see that some set $J\cup L$ of $j+1=2$ or $3$
vertices of $H_{t+1}$ (with at most one in $e_t$, namely that in $L$) is contained
in $\Theta(n)$ deadly $(j+2)$-sets $I\cup J\cup L$. But this violates our assumption that 
condition \eqref{C4} holds, completing the proof of the claim.
\end{proof}

\begin{claim}\label{claim2}
If the candidate $X_{t+1,1}$ satisfies conditions \eqref{C0}--\eqref{C4}, then for $i\in\{2,3\}$
each vertex of $X_{t+1,1}$ is contained in 
at most $2n^{i-\eps/4}$ deadly $(i+1)$-sets {\bf at time} $\mathbf{t+1}$.
\end{claim}
\begin{proof}
The proof is very similar to that of Claim~\ref{claim1}. Adopting similar notation,
if the claim fails there is some $i\in \{2,3\}$ and some singleton set $J=\{u\}\subset X_{t+1,1}$
such that there are $2n^{i-\eps/4}$ $i$-sets $I$ disjoint from $J$ with $I\cup J$
deadly at time $t+1$. Since condition \eqref{C3} holds
and a deadly set is dangerous, at least $n^{i-\eps/4}$ of
these sets $I$ are such that $I\cup J$ is deadly at time $t+1$ but not at time $t$. As before,
since $\Delta(G_{t+1})=O(n^{1-\eps})=o(n^{1-\eps/4})$,
we may find $\Theta(n^{i-\eps/4})>0$ such sets $I$ which are disjoint from $H_{t+1}$ and send no
edges to $H_{t+1}$ in the graph $G_{t+1}$.
Fix any such $I$. Then there is an $r$-set $A\supset I\cup J$ such that $G_{t+1}\cup K_{I\cup J}$
contains all edges in $A$ with (as before) $L=V(H_{t+1})\cap (A\setminus (I\cup J))$ non-empty. (Otherwise
$I\cup J$ was deadly at time $t$ also.) Since $L\subset V(H_{t+1})$, by choice of $I$ there are no
$I$--$L$ edges in $G_{t+1}$, so the $I$--$L$ edges (of which there is at least one) are missing
in $G_{t+1}\cup K_{I\cup J}$ also, a contradiction.
\end{proof}

We now show that for $k=2,3,4$ the probability of the event $\cE_k$ that our first candidate
$X_{t+1,1}$ fails to satisfy condition (C$k$) is $o(1)$. Recall that we are assuming that 
our algorithm succeeded at earlier steps and that neither $\cB_t^1$ nor $\cB_t^2$ holds.
If condition \eqref{C2} fails then for some $0\le j\le 2$ and $1\le i\le r-2$ with $i+j<r$
there is an $i$-element subset $I$ of $X_{t+1,1}$ and a $j$-element subset $J$ of $e_t$
such that $I\cup J$ is dangerous (at time $t$).
For $j=0$, since $\cB_t^1$ does not hold, there are $o(n^i)$ dangerous $i$-sets in $G_t$.
The probability that $X_{t+1,1}$ includes any given one is $O(n^{-i})$, so the probability
that such $I$, $J$ exist with $j=0$ is $o(1)$.
For $j\in\{1,2\}$,  applying Claim~\ref{claim1}
with $t$ replaced by $t-1$ (for $t>0$; for $t=0$ there is no problem),
we see that for each $j$-element subset $J$ of $e_t\subset X_t$ there are $o(n^{i})$
$i$-sets $I$ such that $I\cup J$ is dangerous. It follows that $\Pr(\cE_2)=o(1)$.

Turning to $\cE_3$, fix $j\in \{1,2\}$ and $i\ge 1$ with $i+j<r$.
Call a $j$-set \emph{bad} if it is in at least $n^{i-\eps/4}$ $(i+j)$-sets that are dangerous at time $t$.
Since at most $n^{i+j-\eps/2}$ $(i+j)$-sets are dangerous at time $t$, there are at most
$\binom{i+j}{j}n^{j-\eps/4} = O(n^{j-\eps/4})$
bad $j$-sets. For condition \eqref{C3} to fail, $X_{t+1,1}$ must contain such a bad $j$-set,
an event of probability $O(n^{j-\eps/4}n^{-j})=o(1)$.

Finally, we turn to $\cE_4$; the argument is somewhat similar to that for $\cE_2$.
Firstly, as $\cB_t^1$ does not hold and a deadly set is dangerous, there are at most $(j+1)n^{j-\eps/4}$
$j$-sets $J$ with the property that $J$ is contained in more than $n^{1-\eps/4}$ deadly $(j+1)$-sets
at time $t$. Hence the probability that $X_{t+1,1}$ contains such a $j$-set is $O(n^{j-\eps/4}/n^j)=o(1)$.
For $j=2,3$ it remains to bound the probability that for some vertex $v$ of $e_t$, the
set $X_{t+1,1}$ contains a $(j-1)$-set $J$ such that $J\cup\{v\}$ is in more than $n^{1-\eps/8}$
deadly $(j+1)$-sets. This probability is $o(1)$ as required unless $v$ is in $\Theta(n^{j-1})$
$j$-sets each contained in more than $n^{1-\eps/8}$ deadly $(j+1)$-sets. But if this
holds, $v$ is in $\Theta(n^{j-\eps/8})$ deadly $(j+1)$-sets. Since $v\in e_t\subset X_t$,
this contradicts Claim~\ref{claim2} applied at time $t-1$.
\end{proof}

\section{Deterministic constructions}
\label{sec:deterministic}

In this section we present a deterministic algorithm that constructs an initially infected set which stabilizes after at least $\frac{1}{200}n^{3/2}$ time steps in the $K_5$-bootstrap process. This construction generalizes immediately to any $r \geq 6$ to obtain graphs that stabilize after at least $c_r n^{3/2}$ time steps in the $K_r$-bootstrap process. The result is not as strong as Theorem~\ref{thm:maxTime5}, but the proof, in addition to being deterministic, is much simpler, and may perhaps be of independent interest.

\begin{lemma}
\label{lem:deterministicSystem}
For all $t \leq T = \frac{1}{200}n^{3/2}$ we can find a collection $\AA = \{ A_1, \ldots , A_t \} \subseteq [n]^{(5)}$ and distinct $e_0 = \{1,2\}, e_1, \ldots , e_t \in [n]^{(2)}$ with the following properties:
\begin{enumerate}
\item $e_i = A_i \cap A_{i+1} $ for $i = 1, \ldots ,t-1$,
\item $|A_i \cap A_j| \leq 1 $ if $|i - j| \geq 2$,
\item there are no distinct $u,v,w \in [n]$ and $1 \leq i < j < k \leq t$, such that $\{u,v\} \subset A_i$, $\{u,w\} \subset A_j$ and $\{v,w\} \subset A_k$ (i.e., there are no ``external triangles'' in the set system),
\item for all $u \in [n]$ we have $\deg(u) = |\{A_i \in \AA: u \in A_i\}| \leq \frac{1}{20}\sqrt{n}$.
\end{enumerate}
\end{lemma}
\begin{proof}
Let us define the neighbourhood $N(u)$ of $u \in [n]$ to be
\[
 N(u) = \{ v \in [n] : \mbox{there exists some } A_i \in \AA \text{ such that } u,v \in A \}
\]
Also define the second neighbourhood of $u$ to be
\[
 N^2(u) = \bigcup_{v \in N(u)} N(v).
\]
Notice that $u \in N(u) \subseteq N^2(u)$. Also, if for all $u \in [n]$ we have $\deg(u) \leq \Delta$, then $|N^2(u)| < 25 \Delta^2$.
 
Let us proceed by induction on $t$. Assume that we have sets $A_1, \ldots ,A_{t-1}$, and pairs $e_0 = \{1,2\}, e_1, \ldots , e_{t-1}$, satisfying the lemma. We want to find some $A_{t}$ and $e_t$ while $t < cn^{3/2}$. Let $e_{t-1} = \{u,v \}$. We shall take $A_{t} = e_{t-1} \cup \{w_1,w_2,w_3 \}$ for some distinct $w_1,w_2,w_3 \notin e_{t-1}$, and then set $e_{t} = \{w_1,w_2 \}$. Now we choose the $w_i$, one at a time, so that, for $i = 1,2,3$, the vertex $w_i$ is not contained in
\[
 N^2(u) \cup N^2(v) \cup N^2(w_1) \cup \ldots \cup N^2(w_{i-1}),
\]
where in defining $N^2(\cdot)$
we include the set $A_t=\{u,v,w_1,\ldots,w_{i-1}\}$.
 By the condition on the degrees of all vertices,
\begin{equation}
\label{eq:2ndNeighbourhoodsSmall}
 |N^2(u)| + |N^2(v)| + |N^2(w_1)| + \cdots + |N^2(w_{i-1})| < 100 (\tfrac{1}{20}\sqrt{n}+1)^2 < n/2.
\end{equation}
Choosing $w_i$ that does not belong to any of the above second neighbourhoods ensures that there are no external triangles formed by the addition of $A_t$ and that $|A_t \cap A_j| \leq 1 $ for all $j \leq t-2$.

By \eqref{eq:2ndNeighbourhoodsSmall}, if the maximum degree of a vertex is at most $\frac{1}{20}\sqrt{n}$, then there are at least $\frac{1}{2}n$ possible choices for each of $w_1,w_2,w_3$. Now, for each $i = 1,2,3$, choose $w_i$ to be the vertex with the smallest degree among these choices. If this choice makes the degree of $w_i$ larger than $\frac{1}{20}\sqrt{n}$, then all of these $\frac{1}{2}n$ vertices have degree at least $\frac{1}{20}\sqrt{n}$. Hence there are at least $T = \frac{1}{200}n^{3/2}$ sets in $\AA$. This completes the argument.
\end{proof}

Let us show how to turn the set system $\AA$ into a graph that has a long running time in the $K_5$-bootstrap process. For $T = \frac{1}{200}n^{3/2}$, given $\AA$ and $e_1,\ldots ,e_t$ as defined above, we define $G \subseteq E(K_n)$ to be
\begin{equation}
\label{eq:deterministicGraph}
 G = \left( \bigcup_{i = 1}^T A_{i}^{(2)} \right) \setminus \{ e_1, \ldots , e_T \}.
\end{equation}

Let us run the $K_5$-bootstrap process starting from $G$. We claim that it takes $T$ steps before this process stabilizes. Let $\AA$ be a set system on $[n]$ and let $G \subseteq E(K_n)$. We say that $G$ is \emph{covered} by $\AA$ if every edge $e \in G$ is contained in some set $A_i \in \AA$. We say that $G$ is \emph{simply covered} by $\AA$ if there exists some $A_i \in \AA$ such that $ e \subseteq A$ for every $e \in G$. The following observation follows immediately from the third condition in Lemma~\ref{lem:deterministicSystem}.

\begin{observation}
\label{obs:simpleCovers}
For any distinct $u,v,w \in [n]$, if $\AA$ covers $\{u,v,w\}^{(2)}$ then it simply covers it.
\end{observation}

\begin{lemma}
\label{lem:longRunning}
Let $G_0 = G$ be defined as in \eqref{eq:deterministicGraph}. Then for all $1 \leq t \leq T$ we have $G_{t} \setminus G_{t-1} = \{e_{t}\}$.
\end{lemma}
\begin{proof}
Suppose that the lemma is false and let $t' \geq 1$ be the first time when $G_{t'} \setminus G_{t'-1} \neq \{e_{t'}\}$. Hence, for all $1 \leq t < t'$ we have $G_{t} \setminus G_{t-1} = \{e_{t}\}$. Clearly $A_{t'}$ induces a $K_5^-$ in $G_{t'-1}$ because $e_{t'-1} \in G_{t'-1}$ and, by the definition of $\AA$, all edges $\{e_1, \ldots , e_T\}$ are distinct which implies that $e_{t'} \notin G_{t'-1}$. Hence we have $e_{t'} \in G_{t'} \setminus G_{t'-1}$ and, by the definition of $t'$, we must have some $e \neq e_{t'}$ such that $e \in G_{t'} \setminus G_{t'-1}$.

If $e$ is included at time $t'$ then it appears by virtue of some $K_5^-$ that is contained in $G_{t'-1}$. If there was no $A_i$ that induced all the edges of this $K_5^-$ in $G_{t'-1}$ then, since $G_{t'-1}$ is a subset of $\bigcup_{i = 1}^T A_{i}^{(2)}$, the set system $\AA$ would would cover a triangle that it would not simply cover. That would contradict Observation~\ref{obs:simpleCovers}.

Hence $e = e_k$ for some $k > t'$. However, only two sets in $\AA$ contain $e_k$ namely, $A_k$ and $A_{k+1}$. By the definition of $t'$ we have $e_{k-1}$, $e_k$, $e_{k+1} \notin G_{t'-1}$ therefore both $A_k$ and $A_{k+1}$ induce $\binom{5}{2}-2$ edges in $G_{t'-1}$. Hence it is impossible that $e_k \in G_{t'} \setminus G_{t'-1}$.
\end{proof}

Lemma~\ref{lem:longRunning} immediately implies that the $K_5$-bootstrap process starting from $G$ stabilizes after at least $\frac{1}{200}n^{3/2}$ time steps.

\section{Open problems}
\label{sec:problems}

In this paper we considered the problem of finding the maximum running time of the
$K_r$-bootstrap process. For $r \geq 5$ the exact answer remains to be found and it is the
obvious open problem in this direction.
In the proof of Theorem~\ref{thm:maxTime5} we were a little careless with the $o(1)$ term in the exponent;
a more careful version of the argument will give a slightly stronger bound (presumably $n^{2-\alpha_r}$
divided by some power of $\log n$), but still $o(n^{2-\alpha_r})$.
It seems unlikely that any random construction of this type can go (significantly, if at all)
beyond $n^{2-\alpha_r}$, since
at this point, in a random graph with the relevant number of edges, many missing edges $e$ have
the property that adding $e$ would create a $K_r^-$. It is tempting to think that therefore
$M_r(n)=n^{2-\alpha_r+o(1)}$ for $r\ge 5$, but we have no real reason to believe this.
Instead, since proving any non-trivial upper bound on $M_r(n)$ is open for $r\ge 5$,
we make the following much weaker conjecture.

\begin{conjecture}
 \label{conj:below_n^2}
 For all $r \geq 5$ we have $M_r(n) = o(n^2)$.
\end{conjecture}

The following evidence supports Conjecture~\ref{conj:below_n^2} for $r=5$. We should expect that if the process runs for a long time then the infection spreads in a ``controlled'' way, i.e., most of the time the infection of one edge completes just one infected copy of $K_5^-$, which on the next step again infects only one new edge. For this to happen, the resulting copies of $K_5$ in the infected graph should not be packed ``too tightly'', to prevent the infection spreading out of control. The third condition in Lemma~\ref{lem:deterministicSystem}, about the lack of external triangles, itself implies that we must have $T = o(n^2)$ (see Theorem 1.7 in Erd\H{o}s, Frankl, R\H{o}dl~\cite{erdosFranklRodl}). This condition is more strict than that which a graph on which the infection spreads in a controlled way needs to satisfy. However, if $G_t$ contains two infected copies $A_1, A_2$ of $K_5$ that share an edge $\{u,v\}$, and an additional infected edge $\{w,z\}$ for some $w \in A_1, z \in A_2$ (hence both $\{u,v,w\}$ and $\{u,v,z\}$ induce external triangles), then it is easy to check that the $8$ vertices containing $A_1 \cup A_2$ induce a $K_{8}$ in $G_{t+2}$. This situation seems difficult to avoid if the process is meant to run for at least $\alpha n^2$ steps for some $\alpha > 0$.

\bibliography{mylargebib}

\end{document}